\documentclass{article}
\usepackage{graphicx} 
\usepackage{amsmath, amsthm, amsfonts, amssymb, amscd, bm}
\usepackage{verbatim}
\usepackage{enumerate}
\usepackage{url}

\newtheorem{theorem}{Theorem}[section]
\newtheorem{lemma}[theorem]{Lemma}

\newcommand{\diag}{\mathrm{diag}}
\newcommand{\Tr}{\mathrm{Tr}}

\begin{document}

\title{Solutions of Backward Stochastic Differential Equations on Markov Chains}

\author{Samuel N. Cohen\\ Department of Mathematics, University of Adelaide\\
samuel.cohen@adelaide.edu.au\\ \\
Robert J. Elliott\\ Haskayne School of Business, University of Calgary, Calgary\\ Department of Mathematics, University of Adelaide\\relliott@ucalgary.ca\\http://www.ucalgary.ca/$\sim$relliott}

\date{July 2008}
\maketitle

\begin{abstract}
We consider backward stochastic differential equations (BSDEs) related to finite state, continuous time Markov chains. We show that appropriate solutions exist for arbitrary terminal conditions, and are unique up to sets of measure zero. We do not require the generating functions to be monotonic, instead using only an appropriate Lipschitz continuity condition.
\end{abstract}

\section{Introduction}
Consider a continuous time, finite state Markov chain $X=\{X_t, t\in[0,T]\}$. We identify the states of this process with the unit vectors $e_i$ in $\mathbb{R}^N$, where $N$ is the number of states of the chain. 

We consider stochastic processes defined on the filtered probability space ($\Omega$, $\mathcal{F}$, $\{\mathcal{F}_t\}$, $\mathbb{P}$), where $\{\mathcal{F}_t\}$ is the completed natural filtration generated by the $\sigma$-fields $\mathcal{F}_t=\sigma(\{X_u, u\leq t\}, F\in \mathcal{F}_T:\mathbb{P}(F)=0) $, and $\mathcal{F}=\mathcal{F}_T$. Note that, as $X$ is right-continuous, this filtration is right-continuous. If $A_t$ denotes the rate matrix for $X$ at time $t$, then this chain has the representation
\begin{equation} \label{eq:MarkovChainRep}
X_t= X_0 + \int_{]0,t]} A_u X_u du + M_t
\end{equation}
where $M_t$ is a martingale. (See Appendix B of \cite{Elliott1994}.)

For an $\mathcal{F}_T$ measurable, $\mathbb{R}^K$ valued, $\mathbb{P}$-square integrable random variable $Q$, we shall discuss equations of the form
\begin{equation} \label{eq:GeneralBSDE}
Z_t+ \int_{]t,T]} F(\omega, u, Z_u, Y_u)du+\int_{]t,T]}[G(\omega, u-,Z_{u-})+Y_{u-}] dM_t = Q
\end{equation}
for functions $F: \Omega\times[0,T]\times \mathbb{R}^K\times \mathbb{R}^{K\times N}\rightarrow \mathbb{R}^K$ and $G:\Omega\times[0,T]\times \mathbb{R}^K \rightarrow \mathbb{R}^{K\times N}$. These functions are assumed to be progressively measurable, i.e. $F(., t, Z_t, Y_t)$ and $G(., t, Z_t)$ are $\mathcal{F}_t$ measurable for all $t\in[0,T]$.

 We seek a solution of (\ref{eq:GeneralBSDE}), that is a pair $(Z, Y)$, where $Z$ is an $\mathbb{R}^K$ valued adapted process and $Y$ is an $\mathbb{R}^{K\times N}$ adapted process. We shall address this in four stages: firstly we shall show a martingale representation theorem in this framework, then we shall use this to show the existence and uniqueness of solutions in three stages of increasing complexity in $F$ and $G$. This is essentially the same approach as in \cite{Pardoux1990}; however the details are not the same as our dynamics differ. The key result presented here is Theorem \ref{thm:GeneralBSDESoln}, an existence and uniqueness result. This result can be seen as a special case of that obtained by \cite{El1997a}; however we here present explicit formulae for certain quantities of which they assume the existence, and, by establishing a martingale representation theorem, do not require the existence of a `non-hedgeable' process.
 
 As a side note, observe that (\ref{eq:GeneralBSDE}) is equivalent to 
 \[Z_t+ \int_{]t,T]} F^*(\omega, u, Z_u, Y_u)du+\int_{]t,T]}[G(\omega, u-,Z_{u-})+Y_{u-}(\omega)] dX_t = Q\]
where $F^*(\omega, u, Z_u, Y_u) = F(\omega, u, Z_u, Y_u) - [G(\omega, u-, Z_{u-}) +Y_{u-}(\omega)]A_uX_u(\omega)$. 

\section{Preliminary Concerns}
We note that the optional quadratic variation of $M_t$ is given by the matrix process
\[[M,M]_t = \sum_{0<u\leq t}\Delta M_u \Delta M^*_u.\]
Recalling that $A$ is the rate matrix of the Markov chain $X$, the predictable quadratic variation is
\[ \langle M, M \rangle_t = \int_{]0,t]}[\diag (A_uX_u) - \diag(X_u) A_u^* - A_u\diag(X_u)]du.\]
This can be seen by considering
\[\diag(X_t) = \diag(X_0) + \int_{]0,t]}\diag(A_uX_u)du + \int_{]0,T]} \diag(dM_u)\]
and
\[\begin{split}
\diag(X_t) &= X_tX_t^*\\
&= X_0X_0^* + \int_{]0,t]} X_u X_u^*A_u^* du + \int_{]0,t]} X_{u-}dM_u^* + \int_{]0,t]} A_u X_u X_u^*du \\
&\qquad + \int_{]0,t]} dM_u X_{u-}^* + \sum_{0<u\leq t} \Delta M_u \Delta M_u^*\end{split}.\]
Equating these two gives
\[[M,M]_t = L_t+\int_{]0,t]} [\diag(A_uX_u)-\diag(X_u) A_u^* - A_u \diag(X_u)]du \]
for some martingale $L$. This in turn implies that 
\[ \langle M, M \rangle_t = \int_{]0,t]}[\diag (A_uX_u) - \diag(X_u) A_u^* - A_u\diag(X_u)]du\]
as desired.

Define the following quantities:
\[\begin{split}
\langle C, D \rangle &= \Tr(CD^*),\\
 \|C\|^2 &= \langle C, C \rangle,\\
\langle C, D \rangle_{V} &= \Tr(C[\diag(A_uV)-\diag(V)A_u^* - A_u\diag(V)]D^*),\\
 \|C\|^2_{V} &= \langle C, C \rangle_{V},
\end{split}\]
where $V$ is a (basis) vector in $\mathbb{R}^N$. As the matrix \[\diag(A_uV)-\diag(V)A_u^* - A_u\diag(V)\] is symmetric and positive (semi-)definite, this is a well defined (semi-)norm. One notable feature of this notation is that
\[\int_{]t,T]}\|C\|^2_{X_u}du = \int_{]t,T]}\Tr(Cd\langle M, M\rangle_uC^*).\]
Note that for any adapted process $C$, the quantity $\|C\|^2_X = \{\|C_u\|^2_{X_u}, u\in[0,T]\}$ is an adapted random process with values in $[0,\infty[$. Hence its expectation is well defined for each $u$,
\[E\|C_u\|^2_{X_u} = E[\Tr(C_u[\diag(A_uX_u)-\diag(X_u)A_u^*-A_u\diag(X_u)]C_u^*)].\]

We restrict our attention to when 
 \[E\|Z_u\|^2<+\infty\]  for all $u$ and \[E\|Y_u\|_{X_u}^2<+\infty\] for almost all $u$ (i.e. $E\|Y_u\|<+\infty$ $d\langle M, M\rangle_u$-a.s.).
This, coupled with the Lipschitz conditions placed on $F$ and $G$ will immediately imply
\[E\int_{]0,T]}\|F(\omega, u, Z_u, Y_u)\|^2du <+\infty\] and \[E\int_{]0,T]}\|G(\omega, u, Z_u)\|^2du <+\infty.\]
 This assumption has proven to be important when dealing with similar equations based on Brownian motion; see for example \cite{El1997}.

Note that, as it is only the $u$-left limit $G(\omega, u-, Z_{u-})$ which enters into (\ref{eq:GeneralBSDE}), there is no loss of generality to assume that $G(\omega, u, Z)$ is left continuous in $u$ for each $\omega$ and $Z$. Note also that as $M$ is a semimartingale, $Z$ is c\`adl\`ag and adapted (see \cite[Thm 4.31]{Jacod2003}).

We assume the existence of the left limits of $Y$. Hence, $Y$ must have at most a countable number of discontinuities, and therefore it must be left-continuous at each $t$ except possibly on a $dt$-null set. Hence, if $Y_u$ satisfies (\ref{eq:GeneralBSDE}), then so does $Y_{u-}$,
\[Z_t+ \int_{]t,T]} F(u, Z_u, Y_{u-})du+\int_{]t,T]}[G(u-,Z_{u-})+Y_{u-}] dM_t = Q.\]
Writing $Y^*_t := Y_{t-}$ we have a left-continuous process $Y^*$ which will also satisfy the desired equation, and therefore the writing of the left limits $Y_{t-}$ is unnecessary (as we simply assume our solution is left-continuous).

Given these arguments, we rewrite (\ref{eq:GeneralBSDE}) as
\[Z_t+ \int_{]t,T]} F(u, Z_u, Y_u)du+\int_{]t,T]}[G(u,Z_{u-})+Y_{u}] dM_t = Q,\]

\section{A Martingale Representation Theorem}
\begin{lemma} \label{lem:MartRep}
Any $\mathbb{R}^K$ valued martingale $L$ defined on $(\Omega, \{\mathcal{F}_t\}, \mathbb{P})$ can be represented as a stochastic (in this case Stieltjes) integral with respect to the martingale process $M$, up to equality $\mathbb{P}$-a.s.. This representation is unique up to a $d\langle M,M\rangle_t\times \mathbb{P}$-null set.
\end{lemma}
\begin{proof}
For $i\neq j$, $\langle X_{u-}, e_i\rangle\langle X_u, e_j\rangle=1$ if and only if $X$ jumps from $e_i$ to $e_j$ at time $u$. Then
\[\langle X_{u-}, e_i\rangle\langle X_u, e_j\rangle = \langle X_{u-}, e_i\rangle\langle \Delta X_u, e_j\rangle = \langle X_{u-}, e_i\rangle\langle dX_u, e_j\rangle\] 
and \[N_t^{ij}=\int_{]0,t]}\langle X_{u-}, e_i\rangle\langle dX_u, e_j\rangle\]
is the number of jumps from $e_i$ to $e_j$ in $]0,t]$.

From (\ref{eq:MarkovChainRep}), $dX_t = A_tX_tdt + dM_t$, and so
\[N_t^{ij} = \int_{]0,t]} A_{ji}\langle X_u, e_i\rangle du + Q_t^{ij},\]
where $Q_t^{ij}$ is the compensated jump martingale
\[Q_t^{ij}= \int_{]0,t]}\langle X_{u-}, e_i\rangle\langle dM_u, e_j\rangle.\]

Note that the filtration generated by $X$ is the same as the filtration generated by the processes $N^{ij}, i\neq j, 1\leq i, j \leq N.$ Therefore, for any $\mathbb{R}$-valued square integrable $\{\mathcal{F}_t\}$ martingale $L$, we can write
\[L_t = L_0 + \sum_{i\neq j} \int_{]0,t]} \gamma^{ij}_u dQ^{ij}_u\]
for some predictable processes $\gamma^{ij}_u$. (See \cite{Bremaud1981} for a proof of this.) Define the predictable matrix process $\Gamma$ by 
\[[\Gamma_u]_{ij} = \left\{ \begin{array}{ll} 0 & i=j,\\ \gamma^{ij}_u & i\neq j.\end{array} \right.\]
We can then write 
\[\begin{split}
L_t &= L_0 + \sum_{i\neq j} \int_{]0,t]} \gamma^{ij}_u dQ^{ij}_u\\
&= L_0 + \int_{]0,t]}\sum_{i, j}  \gamma^{ij}_u \langle X_{u-}, e_i\rangle\langle dM_u, e_j\rangle\\
&= L_0 + \int_{]0,t]}X_{u-}^*\Gamma_u dM_u.
\end{split}\]

If we now consider an $\mathbb{R}^K$ valued  $\mathcal{F}_t$ martingale $L$, we can write this as \[(L^1, L^2, ..., L^K)^*,\] where each term is an $\mathbb{R}$-valued martingale. Hence we can write the $L^i$ term as a stochastic integral for some $\Gamma^i$, and so the vector martingale can be written
\begin{equation}\label{eq:MartingaleRepnThm}
L_t = L_0 + \int_{]0,t]}Y_udM_u
\end{equation}
where 
\[Y_u=\left[\begin{array}{c}X_{u-}^*\Gamma^1_u \\X_{u-}^*\Gamma^2_u \\\vdots\\X_{u-}^*\Gamma^K_u \end{array}\right]\]
is a predictable $\mathbb{R}^{K\times N}$ valued matrix process.

Furthermore, this decomposition is unique, in the sense that if 
\[L_t = L_0 + \int_{]0,t]}Y_u^1dM_u = L_0 + \int_{]0,t]}Y_u^2dM_u\]
then
\[0 = \int_{]0,t]}[Y_u^1-Y^2_u]dM_u =: \Phi_t\]
$\mathbb{P}$-a.s. By the Stieltjes product rule,
\[\begin{split}
0 &= \Phi_t \Phi_t^*\\
&= \Phi_0\Phi_0^* + \int_{]0,t]}\Phi_u dM_u [Y^1_u-Y^2_u]^* + \int_{]0,t]}[Y^1_u-Y^2_u]dM_u \Phi_u^* \\&\qquad+ \sum_{0<u\leq t} [Y_u^1-Y^2_u]\Delta M_u \Delta M_u^*[Y^1_u-Y^2_u]^*\\
&= L_t + \int_{]0,t]}[Y_u^1-Y^2_u] d[M,M]_u [Y^1_u-Y^2_u]^*
\end{split}\]
for some martingale $L$. This in turn implies that \[\int_{]0,t]}[Y_u^1-Y^2_u] d\langle M,M \rangle_u [Y^1_u-Y^2_u]^*=0\] $\mathbb{P}$-a.s., and so $Y^1_u = Y^2_u$ $d\langle M,M\rangle_u\times \mathbb{P}$-a.s.
\end{proof}

\begin{lemma} 
For a square-integrable martingale $L$, the process $Y$ satisfies the square integrability condition $E\|Y_t\|_{X_u}^2<+\infty$ $dt$-a.s.
\end{lemma}
\begin{proof}
From Theorem \ref{eq:MartingaleRepnThm}, we have that 
\[L_t = L_0 + \int_{]0,t]}Y_udM_u\]
and we also know that 
\[\sup_{t} E\|L_t\|^2 < +\infty.\]
We can express
\[\|L_t\|^2 = \|L_0\|^2 + 2 \int_{]0,t]}\langle L_u, Y_u dM_u\rangle +\sum_{0<u\leq t}\|Y_u \Delta M_u\|^2\]
and therefore
\[\begin{split}
E\|L_t\|^2 &= E\|L_0\|^2 +E\sum_{0<u\leq t}\|Y_u \Delta M_u\|^2\\
&=E\|L_0\|^2 +E\int_{]0,t]}\|Y_u\|^2_{X_u} du.
\end{split}\]

Hence $E\int_{]0,t]}\|Y_u\|^2_{X_u} du<+\infty$, which implies that $E\|Y_u\|^2_{X_u}<+\infty$ $dt$-a.s. as desired.
\end{proof}

\section{A Simple Case}
\begin{lemma}\label{lem:BSDE1Soln}
Consider a simplified version of (\ref{eq:GeneralBSDE}), namely
\begin{equation}\label{eq:BSDE1}
Z_t+ \int_{]t,T]} F(\omega, u)du+\int_{]t,T]}[G(\omega, u)+Y_{u}] dM_t = Q.
\end{equation}
This equation has a unique solution.
\end{lemma}

\begin{proof}
First, let 
\[Z_t = E[Q-\int_{]t,T]}F(u)du|\mathcal{F}_t]\]
then
\[Z_t -\int_{]0,t]}F(u)du= E[Q-\int_{]0,T]}F(u)du|\mathcal{F}_t]\]
is a square-integrable martingale (note the integrability assumptions on $Q$ and $F$ above), and so by Lemma \ref{lem:MartRep} has a representation
\[Z_t -\int_{]0,t]}F(u)du= \int_{]0,t]}\Gamma_u dM_u\]
for some square-integrable predictable matrix process $\Gamma_u$.
Write $Y_u:=\Gamma_u-G(u)$, so
\[Z_t = \int_{]0,t]}F(u)du+\int_{]0,t]}[G(u)+Y_u] dM_u.\]
By construction, $Z_T = Q$, and a simple substitution argument gives
\[Q =Z_t+ \int_{]t,T]} F(u)du+\int_{]t,T]}[G(u)+Y_{u}] dM_t \]
as desired.

Furthermore, suppose $(Z_t^1, Y_t^1)$ and $(Z_t^2, Y_t^2)$ both solve (\ref{eq:BSDE1}). Then for all $t$
\[Z_t^1-Z_t^2+\int_{]t,T]}[Y_{u}^1-Y_{u}^2] dM_t=0.\]
Taking an $\mathcal{F}_t$ conditional expectation shows that $Z_t^1 = Z_t^2$ $\mathbb{P}$-a.s for each $t$. As $Z^1$ and $Z^2$ are c\`adl\`ag this implies that they are indistinguishable \cite[Lemma 2.21]{Elliott1982}. Uniqueness from the martingale representation above then implies that $Y^1 = Y^2$ $d\langle M,M\rangle_t\times \mathbb{P}$-a.s.
\end{proof}

\section{Increasing complexity} \label{sec:Step2}
We shall now reintroduce the variable $Y_u$ throughout our version of (\ref{eq:GeneralBSDE}), giving
\begin{equation}\label{eq:BSDE2}
Z_t+ \int_{]t,T]} F(u, Y_u)du+\int_{]t,T]}[G(u)+Y_{u}] dM_t = Q.
\end{equation}

Assume that $F$ is Lipschitz continuous in the following way. There exists $c\in \mathbb{R}$ such that for all $u\in[0,T]$
\[\|F(u, Y^1_u)-F(u, Y^2_u)\|^2\leq c^2 \min_i \|Y^1_u-Y^2_u\|^2_{e_i}.\]
Note that this immediately implies
\[E\|F(u, Y^1_u)-F(u, Y^2_u)\|^2\leq c^2 E\|Y^1_u-Y^2_u\|^2_{X_u}.\]

\begin{lemma} \label{lem:BSDE2Soln2}
Under the above Lipschitz condition, (\ref{eq:BSDE2}) has at most one solution, up to indistinguishability for $Z$ and equality a.s. $d\langle M,M \rangle_t \times \mathbb{P}$ for $Y$.
\end{lemma}
\begin{proof}
Suppose $(Z^1_t, Y^1_t)$ and $(Z^2_t, Y^2_t)$ are both solutions to (\ref{eq:BSDE2}). Then  
\[Z^1_t-Z^2_t=Z^1_0-Z^2_0+ \int_{]0,t]} [F(u, Y^1_u)-F(u, Y^2_u)]du+\int_{]0,t]}[Y^1_{u}-Y^2_{u}] dM_t.\]

Using the Stieltjes Chain rule for products,
\[\begin{split}
\|Z_t^1-Z_t^2\|^2 &= \|Z_0^1-Z_0^2\|^2 + 2\int_{]0,t]}\langle Z_u^1-Z_u^2, F(u, Y^1_u)-F(u, Y^2_u)\rangle du\\
&\qquad+2\int_{]0,t]}\langle Z_u^1-Z_u^2, [Y^1_u-Y^2_u]dM_u\rangle +\sum_{0<u\leq t} \|\Delta Z_t^1-\Delta Z_t^2\|^2,
\end{split}\]
and hence, taking expectations and evaluating at $t=T$,
\[\begin{split}
&E\|Z_t^1-Z_t^2\|^2\\ &= - 2\int_{]t,T]}E\langle Z_u^1-Z_u^2, F(u, Y^1_u)-F(u, Y^2_u)\rangle du-E\sum_{t<u\leq T} \|\Delta Z_t^1-\Delta Z_t^2\|^2\\
&= - 2\int_{]t,T]}E\langle Z_u^1-Z_u^2, F(u, Y^1_u)-F(u, Y^2_u)\rangle du-E\sum_{t<u\leq T} \|[Y^1_u-Y^2_u]\Delta M_u\|^2\\
&= - 2\int_{]t,T]}E\langle Z_u^1-Z_u^2, F(u, Y^1_u)-F(u, Y^2_u)\rangle du-\int_{]t,T]}E\|Y^1_u-Y^2_u\|^2_{X_u}du\\
\end{split}\]

We now recall that, for any $x\in \mathbb{R}$ and $\mathbb{R}^N$ inner product $\langle ,\rangle$, 
\[\pm 2\langle a,b\rangle \leq \frac{1}{x^2}\|a\|^2+x^2\|b\|^2\]
(simply expand $0\leq \|\frac{a}{x} \mp xb\|^2$). Using this,
\[\begin{split}
&E\|Z_t^1-Z_t^2\|^2\\ &= \int_{]t,T]}\left[- 2E\langle Z_u^1-Z_u^2, F(u, Y^1_u)-F(u, Y^2_u)\rangle -E\|Y^1_u-Y^2_u\|^2_{X_u}\right]du\\
&\leq \int_{]t,T]}\left[\frac{1}{x^2}E\|Z_u^1-Z_u^2\|^2 +x^2E\|F(u, Y^1_u)-F(u, Y^2_u)\|^2 -E\|Y^1_u-Y^2_u\|^2_{X_u}\right]du\\
&\leq \int_{]t,T]}\left[\frac{1}{x^2}E\|Z_u^1-Z_u^2\|^2 +(x^2c^2-1)E\|Y^1_u-Y^2_u\|^2_{X_u}\right]du\\
\end{split}\]
and hence, through the use of an integrating factor,
\[e^{t/x^2}E\|Z_t^1-Z_t^2\|^2\leq (x^2c^2-1)\int_{]t,T]}e^{u/x^2}E\|Y^1_u-Y^2_u\|^2_{X_u}du.\]

If we now set $x=c^{-1}$ we see that $E\|Z_t^1-Z_t^2\|^2=0$, indicating that $Z^1_t=Z^2_t$ $\mathbb{P}$-a.s. for each $t$. Right continuity then again implies that $Z^1$ and $Z^2$ are indistinguishable.

If we set $x<c^{-1}$, we conclude
\[0\geq\int_{]t,T]}E\|Y^1_u-Y^2_u\|^2_{X_u}du = E\int_{]t,T]}\Tr([Y^1_u-Y^2_u]d\langle M,M\rangle _u[Y^1_u-Y^2_u]^*)\]
and so $Y^1= Y^2$ $d\langle M,M\rangle _u\times \mathbb{P}$-a.s.

Therefore, we claim that the solutions $(Z^1, Y^1)$ and $(Z^2, Y^2)$ are equivalent, and consequently that any solution is unique (up to appropriate sets of measure zero).
\end{proof}

\begin{lemma} \label{lem:BSDE2Soln1}
Under the above Lipschitz condition, (\ref{eq:BSDE2}) has a solution.
\end{lemma}
\begin{proof}
We now wish to demonstrate that for an arbitrary terminal condition $Q$, a solution $(Z_t, Y_t)$ exists. We do so using a Picard-type iteration, where we define recursively $(Z^{n+1}_t, Y^{n+1}_t)$ to be the solution to
\[Z^{n+1}_t + \int_{]t,T]} F(u, Y^n_u)du + \int_{]t,T]}[G(u)+Y^{n+1}_{u}] dM_u = Q.\]
This equation is of the type of (\ref{eq:BSDE1}), and so the existence of a unique solution is guaranteed by Lemma \ref{lem:BSDE1Soln}. We shall show that these iterates form a Cauchy sequence under an appropriate norm, and therefore that their limit exists and solves the desired equation.

As earlier, we can show that 
\[\begin{split}
E\|Z_t^{n+1}-Z_t^n\|^2 &= -2\int_{]t,T]} E\langle Z^{n+1}_u - Z^{n}_u, F(u, Y^n_u)- F(u, Y^{n-1}_u\rangle du \\&\qquad- \int_{]t,T]} E\|Y^{n+1}_u-Y^{n}_u\|^2_{X_u} du\end{split}\]
and so for any $x\in\mathbb{R}$
\begin{equation}\label{eq:P2Int}
\begin{split}0&\leq e^{t/x^2}\|Z_t^{n+1}-Z_t^n\|^2 \\&\leq \int_{]t,T]} e^{t/x^2}c^2 x^2E\|Y_u^n-Y_u^{n-1}\|^2_{X_u} du -\int_{]t,T]} e^{t/x^2}E\|Y_u^{n+1}-Y_u^{n}\|^2_{X_u} du\end{split}
\end{equation}
Setting $x\leq 2^{-1/2}c^{-1}$ we have
\[\begin{split}
\int_{]t,T]} e^{t/x^2}E\|Y_u^{n+1}-Y_u^{n}\|^2_{X_u} du 
&\leq \frac{1}{2}\int_{]t,T]} e^{2tc^2}E\|Y_u^n-Y_u^{n-1}\|^2_{X_u} du\\
&\leq 2^{-n}\int_{]t,T]} e^{2tc^2}E\|Y_u^1-Y_u^0\|^2_{X_u} du
\end{split}\]
and so $Y^n_t$ is a Cauchy sequence under an appropriate norm. By completeness, this implies that a limit exists.  Considering again (\ref{eq:P2Int}), we also see that $Z^n_t$ is a Cauchy sequence, and therefore again a limit exists. Furthermore, it is seen that these limits satisfy (\ref{eq:BSDE2}).
\end{proof}

\section{A General Solution}
We now consider (\ref{eq:GeneralBSDE}) in full generality. We again shall assume Lipschitz continuity on the generators $F$ and $G$; in this case we shall require there to exist $c\in \mathbb{R}$ such that for all $u\in[0,T]$
\[\begin{split}
\|F(u,Z^1_u, Y^1_u)-F(u,Z^2_u, Y^2_u)\|^2&\leq c^2 \min_i \|Y^1_u-Y^2_u\|^2_{e_i}+c^2\|Z^1_u-Z^2_u\|^2\\
\max_i\|G(u,Z^1_{u-})-G(u, Z^2_{u-})\|^2_{e_i} &\leq c^2\|Z^1_{u-}-Z^2_{u-}\|^2.
\end{split}\]
Again, these requirements are only needed to establish the weaker condition
\[\begin{split}
E\|F(u,Z^1_u, Y^1_u)-F(u,Z^2_u, Y^2_u)\|^2&\leq c^2 E\|Y^1_u-Y^2_u\|^2_{X_u}+c^2E\|Z^1_u-Z^2_u\|^2\\
E\|G(u,Z^1_{u-})-G(u, Z^2_{u-})\|^2_{X_u} &\leq c^2E\|Z^1_{u-}-Z^2_{u-}\|^2.
\end{split}\]
We can further reduce the second of these to
\[E\|G(u,Z^1_{u-})-G(u, Z^2_{u-})\|^2 \leq 3ac^2E\|Z^1_{u-}-Z^2_{u-}\|^2\]
as $\|.\|^2_{V}\leq 3a\|.\|^2$, where $a=\max_{u\in[0,T]}\max_{i,j}|[A_u]_{ij}|$.

We also note that as $\Delta Z_u = Y_u\Delta M_u$ and $\Delta M_u=0$ $\mathbb{P}$-a.s. for all $u$, we know that 
\[E\|Z^1_{u-}-Z^2_{u-}\|^2 = E\|Z^1_{u}-Z^2_{u}\|^2.\]
Therefore, our assumption implies the condition
\[E\|G(u,Z^1_{u-})-G(u, Z^2_{u-})\|^2_{X_u} \leq c^2E\|Z^1_{u}-Z^2_{u}\|^2.\]

\begin{lemma} \label{lem:GeneralBSDESolnUnique}
Under the above Lipschitz condition, (\ref{eq:GeneralBSDE}) has at most one solution up to indistinguishability for $Z$ and equality a.s. $d\langle M,M \rangle_t \times \mathbb{P}$ for $Y$.
\end{lemma}

\begin{proof}

As above, suppose $(Z^1_t, Y^1_t)$ and $(Z^2_t, Y^2_t)$ are both solutions to (\ref{eq:GeneralBSDE}). 

Through the same calculations as before, we find that 
\[\begin{split}E\|Z_t^1-Z_t^2\|^2 &= - 2\int_{]0,t]}E\langle Z_u^1-Z_u^2, F(u,Z^1_u, Y^1_u)-F(u,Z^2_u, Y^2_u)\rangle du\\
&\qquad-\int_{]0,t]}E\|G(u, Z_{u-}^1)-G(u, Z_{u-}^2)+Y^1_u-Y^2_u\|^2_{X_u}du\end{split}\]
and so
\[\begin{split}E\|Z_t^1-Z_t^2\|^2 &=\int_{]t,T]}[- 2E\langle Z_u^1-Z_u^2, F(u,Z^1_u, Y^1_u)-F(u,Z^2_u, Y^2_u)\rangle\\
&\qquad-E\|G(u, Z_{u-}^1)-G(u, Z_{u-}^2)\|^2_{X_u}-E\|Y^1_u-Y^2_u\|^2_{X_u}\\
&\qquad-2E\langle G(u, Z_{u-}^1)-G(u, Z_{u-}^2),Y^1_u-Y^2_u\rangle_{X_u}]du\\
&\leq \int_{]t,T]}[\frac{1}{x^2}E\|Z_u^1-Z_u^2\|^2+x^2 \|F(u,Z^1_u, Y^1_u)-F(u,Z^2_u, Y^2_u)\|^2\\
&\qquad+(\frac{1}{y^2}-1)E\|G(u, Z_{u-}^1)-G(u, Z_{u-}^2)\|^2_{X_u}\\
&\qquad+(y^2-1)E\|Y^1_u-Y^2_u\|^2_{X_u}]du\\
\end{split}\]
for all $x,y\in \mathbb{R}$. Now let $x=(2c)^{-1}$ and $y=1/2$. In this case
\[\begin{split}E\|Z_t^1-Z_t^2\|^2 &\leq \int_{]t,T]}[4c^2E\|Z_u^1-Z_u^2\|^2+\frac{1}{4c^2} \|F(u,Z^1_u, Y^1_u)-F(u,Z^2_u, Y^2_u)\|^2\\
&\qquad+3E\|G(u, Z_{u-}^1)-G(u, Z_{u-}^2)\|^2_{X_u}-\frac{1}{2}E\|Y^1_u-Y^2_u\|^2_{X_u}]du\\
&\leq \int_{]t,T]}[4c^2E\|Z_u^1-Z_u^2\|^2+\frac{1}{4}E\|Y^1_u-Y^2_u\|^2_{X_u}+\frac{1}{4}E\|Z^1_u-Z^2_u\|^2\\
&\qquad+3c^2\|Z_u^1-Z_u^2\|^2-\frac{3}{4}E\|Y^1_u-Y^2_u\|^2_{X_u}]du.\end{split}\]
Hence
\[E\|Z_t^1-Z_t^2\|^2 \leq (7c^2 +\frac{1}{4})\int_{]t,T]}E\|Z_u^1-Z_u^2\|^2du\]
and so an application of Gr\"onwall's lemma implies 
\[E\|Z_t^1-Z_t^2\|^2=0,\]
i.e. $Z_t^1=Z_t^2$ $\mathbb{P}$-a.s. for each $t$. Right continuity again implies $Z^1$ and $Z^2$ are indistinguishable.

From the above we can also deduce
\[\int_{]t,T]}E\|Y^1_u-Y^2_u\|^2_{X_u}du\leq 2(7c^2 +\frac{1}{4})\int_{]t,T]}E\|Z_u^1-Z_u^2\|^2du=0\]
and therefore $Y^1=Y^2$ $d\langle M, M\rangle_t\times \mathbb{P}$-a.s. Hence the solution is once again unique.
\end{proof}

\begin{theorem} \label{thm:GeneralBSDESoln}
Under the above Lipschitz condition, (\ref{eq:GeneralBSDE}) has a solution. This solution is then unique up to indistinguishability for $Z$ and equality a.s. $d\langle M,M \rangle_t \times \mathbb{P}$ for $Y$.
\end{theorem}
\begin{proof}
Once again, we shall do this using a Picard-type iteration. We define recursively $(Z^{n+1}, Y^{n+1})$ to be the solution of
\[Z^{n+1}_t + \int_{]t,T]} F(u, Z^n_u, Y^{n+1}_u) du + \int_{]t,T]} G(u, Z^{n}_{u-})+Y_{u}^{n+1} dM_u = Q.\]
Each of these iterates is guaranteed to exist and be unique by Lemmas \ref{lem:BSDE2Soln1} and \ref{lem:BSDE2Soln2}.

Using the same procedure as above we obtain
\begin{equation}\label{eq:s3I}
\begin{split}&E\|Z^{n+1}_t -Z^{n}_t\|^2+\frac{1}{2}\int_{]t,T]}E\|Y^{n+1}_u-Y^{n}_u\|^2_{X_u}du\\
&\qquad \leq (4c^2+1)\int_{]t,T]}\left[E\|Z^{n+1}_u - Z^{n}_u\|^2 +E\|Z^{n}_u - Z^{n-1}_u\|^2\right]du.\end{split}
\end{equation}
This implies
\[E\|Z^{n+1}_t -Z^{n}_t\|^2 \leq (4c^2+1)\int_{]t,T]}[E\|Z^{n+1}_u - Z^{n}_u\|^2 +E\|Z^{n}_u - Z^{n-1}_u\|^2] du.\]
Rearrangement and integration gives
\[\int_{]t,T]}E\|Z^{n+1}_u - Z^{n}_u\|^2du \leq (4c^2+1)\int_{]t,T]}e^{(4c^2+1)(u-t)}\int_{]u,T]}E\|Z^{n}_s - Z^{n-1}_s\|^2dsdu\]
and hence
\[\int_{]t,T]}E\|Z^{n+1}_u -Z^{n}_u\|^2du \leq \frac{[T(4c^2+1)e^{T(4c^2+1)}]^n}{n!}\int_{]t,T]}E\|Z^{1}_u -Z^{0}_u\|^2du.\]
From this, we know that $\{Z^n_t\}$ forms a Cauchy sequence, except possibly on some $dt$-null set, and so the $\mathbb{P}$-a.s. limit exists by completeness and right continuity. Similarly we can then rearrange (\ref{eq:s3I}) to give
\[\int_{]t,T]}E\|Y^{n+1}_u-Y^{n}_u\|^2_{X_u}du\leq 2(4c^2+1)\int_{]t,T]}E\|Z^{n+1}_u - Z^{n}_u\|^2 +E\|Z^{n}_u - Z^{n-1}_u\|^2du,\]
again showing that $\{Y^n\}$ forms a Cauchy sequence, and so its limit exists. Again, these limits can be seen to satisfy (\ref{eq:GeneralBSDE}). The desired uniqueness properties follow from Lemma \ref{lem:GeneralBSDESolnUnique}.
\end{proof}

\section{Conclusion}
We have shown here that for an equation of the form of (\ref{eq:GeneralBSDE}), there exists a square integrable solution $(Z, Y)$, where $Z$ is an $\mathbb{R}^K$ valued adapted process and $Y$ is an $\mathbb{R}^{K\times N}$ valued adapted process. We have shown that, and without loss of generality, $Y$ can be taken to be left-continuous. We have also shown that thise solution is unique up to $\mathbb{P}$-indistinguishability for $Z$ and equality $d\langle M, M\rangle_u\times \mathbb{P}$-a.s. for $Y$.

\par\bigskip\noindent
{\bf Acknowledgment.}  The authors wish to thank the Social Sciences and Humanities Research 
Council of Canada and the Australian Research Council for support.

\bibliographystyle{amsplain}

\end{document}